\documentclass[11pt,a4paper,twoside]{article}
\usepackage{latexsym,amsfonts,amsmath, amsthm, amssymb}
\usepackage{fullpage}
\usepackage{xcolor,bm, bbm}
\usepackage[utf8x]{inputenc}
\usepackage{array}
\usepackage{mathrsfs}

\usepackage{fourier}

\newcommand{\R}{\mathbb R}
\newcommand{\N}{\mathbb N}

\newcommand{\E}{\mathbb E}
\newcommand{\Pro}{\mathbb P}

\newcommand{\Var}{\mathrm{Var}}

\newcommand{\vol}{\mathrm{vol}}

\def\dint{\textup{d}}
\newcommand{\SSS}{\ensuremath{{\mathbb S}}}

\newcommand{\B}{\ensuremath{{\mathbb B}}}
\newcommand{\D}{\ensuremath{{\mathbb D}}}

\DeclareMathOperator{\id}{id}

\newtheorem{thm}{Theorem}[section]

\newtheorem{lemma}[thm]{Lemma}

\newtheorem{proposition}[thm]{Proposition}

\newtheorem{thmalpha}{Theorem}

{
\theoremstyle{definition}

\newtheorem{rmk}[thm]{Remark}

}

\allowdisplaybreaks
\begin{document}

\title{\bf The maximum entropy principle and volumetric properties of Orlicz balls}

\medskip

\author{Zakhar Kabluchko and Joscha Prochno}



\date{}

\maketitle

\begin{abstract}
\small
We study the precise asymptotic volume of balls in Orlicz spaces and show that the volume of the intersection of two Orlicz balls undergoes a phase transition when the dimension of the ambient space tends to infinity. This generalizes a result of Schechtman and Schmuckenschl\"ager [GAFA, Lecture notes in Math. 1469 (1991), 174--178] for $\ell_p^d$-balls. As another application, we determine the precise asymptotic volume ratio for $2$-concave Orlicz spaces $\ell_M^d$. Our method rests on ideas from statistical mechanics and large deviations theory, more precisely the maximum entropy or Gibbs principle for non-interacting particles, and presents a natural approach and fresh perspective to such geometric and volumetric questions. In particular, our approach explains how the $p$-generalized Gaussian distribution occurs in problems related to the geometry of $\ell_p^d$-balls, which are Orlicz balls when the Orlicz function is $M(t) = |t|^p$.
\medspace
\end{abstract}


\section{Introduction and main results}

Let $d\in \N$, $1\leq r \leq \infty$, and denote by $\D_r^d$ the volume normalized ball in the space $\ell_r^d$. In \cite{SS1991}, Schechtman and Schmuckenschl\"ager studied the asymptotic behavior of the volume of the intersection of a volume normalized ball $\D_p^d$ with a $t$-multiple of a volume normalized ball $\D_q^d$ as the dimension of the ambient space tends to infinity, where $0<p\leq \infty$ and $0<q< \infty$. What they discovered is a threshold phenomenon which says that, for all $t>0$,
\begin{align}\label{thm:schechtman-schmuckenschlaeger}
\vol_d\big(\D_p^d \cap t\D_q^d\big)
\stackrel{d\to\infty}{\longrightarrow}
\begin{cases}
0 & :\, t A_{p,q}<1 \\
1 & :\, t A_{p,q}>1,
\end{cases}
\end{align}
where
\[
A_{p,q} =  \begin{cases}
\frac {\Gamma(1+{1\over p})^{1+{1/q}}}{\Gamma(1+{1\over q})\Gamma({q+1\over p})^{1/q}}\,e^{{1/p}-{1/q}}\,\big({p\over q}\big)^{1/q} &: p<\infty\\
\Gamma(1+{1\over q})^{-1} \big(\frac {q+1}{qe}\big)^{1/q} &: p=\infty.
\end{cases}
\]
The proof heavily rests on a probabilistic representation going back independently to Schechtman and Zinn \cite{SchechtmanZinn} and Rachev and R\"uschendorf \cite{RR1991}. This representation says that the uniform distribution on an $\ell_r^d$-ball can be obtained by considering a sequence $Z_1,\dots,Z_d$ of independent $r$-generalized Gaussians having Lebesgue density
\[
x\mapsto \frac{1}{2r^{1/r}\Gamma(1+1/r)}e^{-|x|^r/r}
\]
and letting
\begin{align}\label{eq:probabilistic representation}
X& := U^{1/d}\frac{(Z_1,\dots,Z_d)}{\|(Z_1,\dots,Z_d)\|_r}
\end{align}
with $U$ uniformly distributed on $[0,1]$ and independent of the $Z_i$'s. The volume of the intersection of balls may be written as the probability that the $\ell_q$-norm of a point uniformly distributed in $\D_p^d$ is bounded above by $tr_q$, where $r_q:=r_q(d)$ is the radius of $\D_q^d$.
So instead of working with a random vector with dependent coordinates (at least when $r<\infty$) the probabilistic representation of the uniform distribution allows one to go over to a random vector with independent ones. The next key ingredients in the proof of the phase transition in \eqref{thm:schechtman-schmuckenschlaeger} are the law of large numbers and the knowledge of the precise asymptotic volumes of $\ell_r^d$-balls. The latter are known at least since Dirichlet \cite{D1839}. For a better understanding, let us briefly sketch the proof for the simple case where $p=\infty$ and $0<q<\infty$. Consider a random vector $Z=(Z_1,\dots,Z_d)$ uniformly distributed on $[-1/2,1/2]^d$, i.e., $Z$ has independent coordinates uniformly distributed on $[-1/2,1/2]$. Then the volume of the intersection can be rewritten as follows,
\[
\vol_d\big(\D_p^d \cap t\D_q^d\big) = \Pro\big[\|Z\|_q \leq t r_q\big] = \Pro\Bigg[\Big(\frac{1}{d}\sum_{i=1}^d|Z_i|^q\Big)^{1/q} \leq t \frac{r_q}{d^{1/q}}\Bigg].
\]
It is then just left to observe that as the dimension tends to infinity, by Stirling's formula $r_q/d^{1/q}$ converges to some explicit number while by the strong law of large numbers the empirical average converges to the expectation of $|Z_1|^q$, which can be computed explicitly.

An inspection of the proof shows that a law of large numbers is however not enough to determine the asymptotic behavior in \eqref{thm:schechtman-schmuckenschlaeger} at the threshold $tA_{p,q}=1$. This problem remained open for a decade until resolved by Schmuckenschl\"ager \cite{Schmu2001} proving a central limit theorem for $\ell_q$-norms of points chosen uniformly at random in $\ell_p^d$-balls. How a central limit theorem helps to answer this question can be seen rather easily in the simple case $p=\infty$ and $0<q<\infty$, where after a different normalization than above the classical central limit theorem gives the answer. More precisely, one sees that when $tA_{p,q}=1$ the limit in \eqref{thm:schechtman-schmuckenschlaeger} is equal to $1/2$. In \cite{KPT2019_I,KPT2019_II}, Kabluchko, Prochno, and Th\"ale extended the previous results in various directions. They proved a multivariate central limit theorem for $\ell_q$-norms of random vectors in $\ell_p^d$-balls and beyond those Gaussian fluctuations they also determined the moderate and large deviations behavior. Applications of those results include an asymptotic version of a result of Schechtman and Zinn \cite[Subsection 2.5]{KPT2019_I}, a demonstration that in the critical case arbitrary limits in $(0,1)$ can occur \cite[Corollary 2.2]{KPT2019_I}, a result on the volume of intersections of neighboring and multiple balls \cite[Corollary 2.3]{KPT2019_I}, where the answer in the critical case is \emph{not} $2^{-d}$ as may be expected, a comparison of random and non-random projections of $\ell_p^d$-balls to lower-dimensional subspaces \cite[Section 2]{KPT2019_II}, and several other applications. A non-commutative version of the Schechtman-Schmuckenschl\"ager result for unit balls in classical random matrix ensembles was recently proved by Kabluchko, Prochno, and Th\"ale in \cite{KPT2020_intersections}. We also refer the reader to the recent survey \cite{PTT2020}.

In this paper, we generalize the result of Schechtman and Schmuckenschl\"ager in a different direction. Most of the previously mentioned results are obtained in the setting of $\ell_p^d$-balls, where it was crucial to have a probabilistic representation of the form presented in \eqref{eq:probabilistic representation} or a more general one from \cite{BartheGuedonEtAl} due to Barthe, Gu\'edon, Mendelson, and Naor, which is still restricted to $\ell_p^d$-balls. Here we study the more general setting of balls in classical, finite-dimensional Orlicz spaces, named after Polish mathematician W\l adys\l aw Orlicz.  Those spaces are natural generalizations of $\ell_p$-spaces and belong to the important class of symmetric Banach sequence spaces. Orlicz spaces are intensively studied in the functional analysis literature and we refer the reader to \cite{HKTJ2006, K1984, KMW2011, KS1985,PS2012,RS1988,S1995} and the references cited therein. Here we study the volumetric properties of balls in Orlicz spaces and obtain a Schechtman-Schmuckenschl\"ager result in this generalized framework. In fact, once we have determined the asymptotic volume of Orlicz balls, we can also compute the precise asymptotic volume ratio of ($2$-concave) Orlicz spaces. This quantity, which will be introduced when we present our main results, is deeply rooted in the geometry of Banach spaces and connected to several other quantities, such as the cotype-$2$ constant. In view of what we have explained before, at its core this generalized setting requires -- modulo other technicalities -- dealing with two problems. First, the absence of a Schechtman--Zinn-type probabilistic representation. Second, one needs to determine the precise asymptotic volume of unit balls in Orlicz spaces. Both problems can be overcome by taking the right perspective. In fact, it seems that a natural way to look at this problem is from the statistical mechanics and large deviations point of view, using the maximum entropy principle in the framework of non-interacting particles. This principle leads to a Gibbs distribution naturally associated with our problem. In particular, this explains how the $p$-generalized Gaussian distribution appears in this type of problems, providing a deeper structural insight and fresh perspective that we think will be useful in other geometric problems.

Before we present the main results of this paper, let us briefly explain the maximum entropy principle that leads to the distributions naturally associated with our problem.

\subsection{Main results}
We shall now present the main results of this paper, starting with the asymptotic (logarithmic) volume before we present the phase transition for the volume of intersections of Orlicz balls. In fact, we shall prove both the formula for the asymptotic logarithmic volume and for the precise asymptotic volume. The reason is that the former result follows from an exponential tilting technique coupled with the classical central limit theorem and gives some structural insight which is lost in the short proof of the precise asymptotic volume in which we use ideas and results from a paper on sharp Cram\'er large deviations by Petrov \cite{P1965}. In addition, one could prove both results in more general settings requiring in each case something weaker than $M$ be an Orlicz function (i.e., $M(0)=0$, $M(t)>0$ for $t\neq 0$, and $M$ is even and convex) and what is needed can be seen from the respective proofs. In what follows, for $R\in(0,\infty)$, let us denote by $B_M^d(dR)$ the Orlicz ball
\[
B_M^d(dR) = \Bigg\{x=(x_i)_{i=1}^d \in\R^d\,:\, \sum_{i=1}^d M(x_i) \leq dR \Bigg\}.
\]

\begin{thmalpha}\label{thm:log-volume orlicz}
Let $d\in\N$, $R\in(0,\infty)$, and $M$ be an Orlicz function. Then, as $d\to\infty$,
\[
\vol_d\big(B_M^d(dR)\big)^{1/d} \to e^{\varphi(\alpha_*)-\alpha_* R},
\]
i.e., on a logarithmic scale, we have
\[
\lim_{d\to\infty} \frac 1d \log \,\vol_d\big(B_M^d(dR)\big) = \varphi(\alpha_*)-\alpha_* R,
\]
and the  precise asymptotic volume is given by
\[
\vol_d\big(B_M^d(dR)\big) \sim \frac{1}{|\alpha_*| \sqrt{2\pi d\, \sigma_*^2}  }e^{d[\varphi(\alpha_*)-\alpha_* R]},
\]
where $\varphi:(-\infty,0)\to \R$ is given by $\varphi(\alpha) = \log \int_{\R}e^{\alpha M(x)}\,\dint x$ and $\alpha_*<0$ is chosen in such a way that $\varphi'(\alpha_*) = R$.
\end{thmalpha}

The next result determines the asymptotic behavior of the volume of intersections of two Orlicz balls when the dimension tends to infinity and generalizes the work of Schechtman and Schmuckenschl\"ager \cite{SS1991}. To merely obtain the phase transition it is enough to know the asymptotic logarithmic volume of Orlicz balls. However, we believe that the precise asymptotics are needed to deal with the critical case at the threshold, a problem we are currently investigating. Before we state the result notice that, since $M$ is an Orlicz function, we have, for all $a>0$, that
\begin{align}\label{eq:integral finite}
\int_{\R} e^{-aM(x)}\,\dint x <+\infty.
\end{align}
This follows directly if we let $c:=M(1)>0$ and observe that $M(x) \geq cx$ for all $x\geq 1$ because of the convexity assumption. Moreover, by the Leibniz integral rule, the integral in \eqref{eq:integral finite} is infinitely differentiable on $(0,\infty)$ as function in the variable $a$.

\begin{thmalpha}\label{thm:dichotomy}
Let $M_1$ and $M_2$ be two Orlicz functions and $R_1,R_2\in(0,\infty)$. Consider
\[
\varphi_1:=\varphi_{M_1}:(-\infty,0)\to\R,\qquad \varphi_1(\alpha) = \log \int_{\R}e^{\alpha M_1(x)}\,\dint x
\]
and choose $\alpha_*<0$ such that $\varphi_1'(\alpha_*) = R_1$. Define the Gibbs density
\[
p_1(x) := e^{\alpha_*M_1(x) - \varphi_1(\alpha_*)}, \qquad x\in\R.
\]
Then, we have
\[
\frac{\vol_d\big(B_{M_{1}}^d(dR_1) \cap B_{M_2}^d(dR_2)\big)}{\vol_d\big(B_{M_1}^d(dR_1)\big)} \stackrel{d\to\infty}{\longrightarrow}
\begin{cases}
0 & :\, \int_{\R}M_2(x)p_1(x)\,\dint x>R_2 \\
1 & :\, \int_{\R}M_2(x)p_1(x)\,\dint x<R_2,
\end{cases}
\]
and the speed of convergence is exponential.
\end{thmalpha}

The last result of this manuscript concerns the precise asymptotic volume ratio of $2$-concave Orlicz spaces, i.e., those Orlicz spaces defined by an Orlicz function $M$ for which $M\circ\sqrt{|\cdot|}$ is concave. The volume ratio is an important quantity related to the geometry of finite-dimensional Banach spaces and defined as follows. Let $d\in\N$ and $K$ be a $d$-dimensional convex body. The volume ratio ${\rm vr}(K)$ of $K$ is defined as
$$
{\rm vr}(K) := \inf\left({\vol_d(K)\over\vol_d(\mathscr{E})}\right)^{1/d},
$$
where the infimum is taken over all ellipsoids $\mathscr{E}$ which are contained in $K$. If $K$ is the unit ball of a $d$-dimensional normed space $X$, then one also speaks of the volume ratio of $X$. The concept of volume ratio is a powerful one, having its origin in the works of Szarek \cite{S1978}, and Szarek and Tomczak-Jaegermann \cite{ST1980}. It lies at the very heart of a famous result of Ka\v{s}in on nearly Euclidean decompositions of $\ell_1^n$ and is also connected to the so-called Rademacher cotype-$2$ constant as is known from a deep result of Bourgain and Milman \cite{BM1987}. The volume ratio has been determined for various Banach spaces and we refer the reader to, e.g., \cite{DM2006, DP2009, GPSTJW2017, Sch1982}. In particular, we refer to \cite{KPT2020_volume_ratio} where the precise asymptotic volume ratio of Schatten $p$-classes, the non-commutative versions of $\ell_p$-spaces, has been computed quite recently based on logarithmic potential theory, which can be viewed as a subfield of statistical mechanics. This time the route to the precise asymptotics is based again on an idea from statistical mechanics, the principle of maximum entropy.

\begin{thmalpha}\label{thm:asymptotic volume ratio}
Let $M$ be a $2$-concave Orlicz function. Then, as $d\to\infty$, we have
\[
\lim_{d\to\infty} {\rm vr}\big(B_M^d(d)\big) = \frac{1}{\sqrt{2\pi e}M^{-1}(1)}e^{\varphi(\alpha_*) - \alpha_*},
\]
where $\varphi:(-\infty,0)\to \R$ is given by $\varphi(\alpha) = \log \int_{\R}e^{\alpha M(x)}\,\dint x$ and $\alpha_*<0$ is chosen in such a way that $\varphi'(\alpha_*) = 1$.
\end{thmalpha}

\subsection{The maximum entropy principle \& Gibbs measures}
Let us explain here how the distribution $p_1(x)$ that plays a central role in Theorems \ref{thm:log-volume orlicz}, \ref{thm:dichotomy}, and \ref{thm:asymptotic volume ratio} naturally appears through what is known as the maximum entropy principle. Although the argumentation is not mathematically rigorous, it shows how distributions of Gibbs-type appear. We follow the exposition in \cite{RAS2015} and also refer the reader to \cite[Section 7.3]{DZ2010} and \cite[Section III]{E2006} for detailed expositions regarding micro-canonical and canonical ensembles.

Consider a sequence of independent and identically distributed random variables $Y_1,Y_2,\dots$ taking values in some Polish space $E$ and having distribution $\lambda\in\mathscr M_1(E)$, where $\mathscr M_1(E)$ is the space of probability measures on $E$ which we equip with the weak topology. With this topology, $\mathscr M_1(E)$ becomes a Polish space itself. For $d\in\N$, we denote by $L_d:=L_d^Y\in\mathscr M_1(E)$ the empirical measure associated with the $Y_i$'s, i.e.,
\[
L_d := \frac{1}{d}\sum_{i=1}^d \delta_{Y_i}.
\]
This measure is obviously a random probability measure. In the setting of Sanov's theorem (see, e.g., \cite[Section 5.2]{RAS2015}) we know that, as $d$ tends to infinity, $L_d\to \lambda$ almost surely at an exponential rate. If we consider a set $C$ of probability measures whose closure does not contain the measure $\lambda$, then by the law of large numbers, $\Pro[L_d \in C]\to 0$ as $d\to\infty$. The maximum entropy principle helps us to understand the case, where we condition on the rare event that $L_d$ remains in $C$. Roughly speaking and under certain assumptions, $L_d$ converges to the element in the set $C$ that minimizes the relative entropy (or Kullback-Leibler divergence) $H(\cdot|\lambda)$, and so \emph{maximizes} thermodynamic entropy. Recall that for probability measures $\nu,\mu\in\mathscr M_1(E)$,
\[
H(\nu|\mu) :=
\begin{cases}
\int_E p \log p \,\dint \mu & :\, p=\frac{\dint \nu}{\dint \mu}\,\,\text{ exists}\\
+\infty &\, \text{otherwise}.
\end{cases}
\]
Being a bit more formal, the maximum entropy principle states that if $C\subset \mathscr M_1(E)$ is closed, convex and satisfies
\[
\inf_{\nu\in C} H(\nu|\lambda) = \inf_{\nu\in C^{\circ}} H(\nu|\lambda) <+\infty,
\]
where $C^{\circ}$ denotes the interior of $C$, then there is a unique measure $\nu_*\in C$ minimizing $H(\cdot|\lambda)$ over the set $C$. Moreover, the conditional distributions of $L_d$ converge weakly, as $d\to\infty$, to $\delta_{\nu_*}$, i.e.,
\[
\lim_{d\to\infty} \Pro[L_d\in \cdot \,|\, L_d\in C] = \delta_{\nu_*}(\cdot)
\]
in the weak topology on $\mathscr M_1(\mathscr M_1(E))$ generated by $\mathscr C_b(\mathscr M_1(E))$, which is the space of bounded continuous and real-valued functions on $\mathscr M_1(E)$. Furthermore, one can show that for any $k\in\N$ the conditional distribution of $Y_k$ (conditioned on $L_d$ being in $C$) converges weakly to the relative entropy minimizing measure $\nu_*$. An application of the maximum entropy principle now shows how a Gibbs measure arises as limiting distribution, which is exactly what happens in the case of Orlicz balls.

So let $\mathscr H:E\to \R$  be a function (often referred to as Hamiltonian or energy) and consider $\overline{\mathscr H}_d := \frac{1}{d}\sum_{i=1}^d \mathscr H(Y_i)$, an average energy. Moreover, define for $R < \E_\lambda[\mathscr H]$ a set
\[
C := \Big\{\nu \in\mathscr M_1(E)\,:\, \E_\nu[\mathscr H]\leq R \Big\}.
\]
If the set $C$ satisfies the assumptions of the maximum entropy principle, then there exists a unique probability measure $\mu_{*}\in C$ minimizing the relative entropy $H(\cdot|\lambda)$ over $C$. Explicitly, it is given as the following Gibbs measure at inverse temperature $\alpha_*$:
\[
\mu_{*}(\dint x) = \frac{e^{-\alpha_*\mathscr H(x)}}{\int_E e^{-\alpha_*\mathscr H(x)}\,\lambda(\dint x)} \lambda(\dint x),
\]
where $\alpha_*>0$ is such that $\E_{\mu_{*}}[\mathscr H]=R$. So wrapping everything up, the maximum entropy principle says in this case that for each $k\in\N$ fixed,
\[
\lim_{d\to\infty} \Pro\Big[Y_k \in \cdot\,\Big|\Big.\, \overline{\mathscr H}_d\leq R\Big] = \mu_{*}.
\]

Now let us explain how this relates to our situation. Please note that this derivation is not mathematically rigorous, one the reasons being that in our setting  $\lambda$ is the Lebesgue measure, which is infinite.   Let $M$ be an Orlicz function and consider, for large $d$, random ``variables'' $Y_1,Y_2,\dots,Y_d$ uniformly ``distributed'' according to the infinite Lebesgue measure $\lambda$. We are interested in the volume of the Orlicz ball
\[
B_M^d(d) = \Bigg\{ (x_1,\dots,x_d) \in\R^d \,:\, \sum_{i=1}^d M(x_i) \leq d \Bigg\}.
\]
Conditioning on $Y=(Y_1,\dots,Y_d)$ being in $B_M^d(d)$ yields the uniform distribution on $B_M^d(d)$, because for any measurable subset $A$ of $\R^d$,
\[
\Pro\Big[Y\in A\,\Big|\Big.\, Y\in B_M^d(d)\Big] = \frac{\Pro[Y\in A\cap B_M^d(d)]}{\Pro[Y\in B_M^d(d)]} = \frac{\vol_d(A\cap B_M^d(d))}{\vol_d(B_M^d(d))}.
\]
Coming back to the maximum entropy principle, where $E=\R$, the Hamiltonian is given by the Orlicz function $M$, and $R=1$ (which is smaller than $\E_\lambda[M] = +\infty$), we have, roughly speaking, for any fixed $k\in\{1,\dots,d \}$  that
\[
\Pro\Big[Y_k\in \cdot\,\Big|\Big.\, \overline{\mathscr H}_d \leq 1 \Big]
=
\Pro\Big[Y_k \in \cdot \,\Big|\Big.\, \sum_{i=1}^d M(Y_i) \leq d \Big] \approx \mu_{*},
\]
where $\alpha_*>0$ is chosen such that $\E_{\mu_{*}}[M]=1$. So, under the ``energy constraint'' that $Y$ lies in an Orlicz ball, asymptotically the coordinates of $Y$ follows a Gibbs distribution $\mu_{*}$. So when studying random vectors in Orlicz balls or volumetric properties of Orlicz balls, then those Gibbs distributions provide the right probabilistic set-up for investigations.
Let us remark that a version of Sanov's theorem with infinite underlying measure $\lambda$ has been obtained in~\cite{BS2016}.

\vskip 5mm

The rest of the paper is organized as follows. In Section \ref{sec:prelim}, we present notions and notation as well as some background regarding Orlicz spaces and sharp Cram\'er large deviations. Then, in Section \ref{sec:asymptoticvolume}, we present the computation of the asymptotic ($\log$-)volume of Orlicz balls. After having computed asymptotic volumes, in Section \ref{sec:volumeintersection}, we can deal with the case of the asymptotic volume of the intersection of two Orlicz balls. Last but not least, in Section \ref{sec:asymptotic volume ratio}, we present the proof for the asymptotic formula of the volume ratio of $2$-concave Orlicz spaces.

\section{Preliminaries}\label{sec:prelim}

In this section, we present some notation and background material needed throughout the paper.

\subsection{Notation}

We shall denote by $\R^d$ the $d$-dimensional Euclidean space. The interior of a set $A\subset \R^d$ shall be denoted by $A^{\circ}$ and its closure by $\overline A$. When we speak of volume in $d$-dimensional space, denoted by $\vol_d$, then we simply mean the $d$-dimensional Lebesgue measure. For two sequences $(a_d)_{d\in\N}$ and $(b_d)_{d\in\N}$ of real numbers, we write $a_d \sim b_d$ if $\lim_{d\to\infty} \frac{a_d}{b_d}=1$.

\subsection{Orlicz spaces}

Let us recall that a convex function $M:\R\to \R$ is said to be an Orlicz function if $M(t) = M(-t)$, $M(0)=0$, and $M(t)>0$ for $t\neq 0$.
The functional
\[
\|(x_1,\dots,x_d)\|_{M} := \inf \left \{ \rho > 0 \,:\, \sum_{i=1}^dM\Big(\frac{|x_i|}{\rho}\Big)  \le 1 \right\}
\]
is a norm on $\R^d$, known as Luxemburg norm, named after W.~A.~J.~Luxemburg~\cite{Lux1955}. We now define the Orlicz space $\ell_M^d$ to be $\R^d$ equipped with this norm and denote by
\[
\B_M^d := \Big\{x=(x_i)_{i=1}^d\in\R^d\,:\, \|x\|_M \leq 1 \Big\}
\]
the unit ball in this space. Those spaces naturally generalize the classical $\ell_p^d$-spaces and belong to the class of $1$-symmetric Banach spaces. One commonly just speaks of Orlicz functions, Orlicz norms, and Orlicz spaces. An introduction to the theory of Orlicz spaces can be found in \cite{KR1961}.

In the next lemma, we collect the simple observation that $\B_M^d$ coincides with the set
\[
B_M^d := \Bigg\{x=(x_i)_{i=1}^d \in\R^d\,:\, \sum_{i=1}^d M(x_i) \leq 1 \Bigg\},
\]
which simplifies some computations because we do not need to work with the infimum. For the sake of completeness, we provide a proof of this fact.

\begin{lemma}\label{lem:equality of unit balls}
Let $d\in\N$ and $M$ be an Orlicz function. Then $\B_M^d = B_M^d$.
\end{lemma}
\begin{proof}
Obviously the $0$-vector is contained in both sets and we may assume from now on that $x\neq 0$.

First assume that $x=(x_i)_{i=1}^d\in \B_M^d$. Then
\[
1\geq \|x\|_M = \inf\Bigg\{\rho>0\,:\,\sum_{i=1}^d M(x_i/\rho) \leq 1 \Bigg\}.
\]
Assume that $\sum_{i=1}^d M(x_i) > 1$. Then, because of the continuity of $M$, there exists $\varepsilon=\varepsilon(x)\in(0,\infty)$ such that $\sum_{i=1}^dM(x_i/(1+\varepsilon))>1$. On the other hand, for this $\varepsilon\in(0,\infty)$ there exists $\rho_0\in(0,\infty)$ such that $\rho_0< \|x\|_M+\varepsilon \leq 1+\varepsilon$ and $\sum_{i=1}^d M(x_i/\rho_0) \leq 1$. But then, since $M$ is increasing, we obtain the contradiction
\[
1< \sum_{i=1}^dM(x_i/(1+\varepsilon)) \leq \sum_{i=1}^d M(x_i/\rho_0) \leq 1.
\]
So $\sum_{i=1}^dM(x_i) \leq 1$, which means that $x\in B_M^d$.

Now let $x=(x_i)_{i=1}^d\in B_M^d$, i.e., $\sum_{i=1}^d M(x_i) \leq 1$. We consider two cases. First assume that $\sum_{i=1}^d M(x_i) < 1$. Then, because of the continuity of $M$, there exists $\varepsilon=\varepsilon(x)\in(0,\infty)$ such that
\[
\sum_{i=1}^d M(x_i/(1-\varepsilon)) \leq 1.
\]
This means $\|x\|_M \leq 1-\varepsilon \leq 1$ and so $x\in \B_M^d$. Now let $\sum_{i=1}^d M(x_i) = 1$ and assume that $\|x\|_M<1$. Then there exists some $\varepsilon=\varepsilon(x)\in(0,\infty)$ such that $\|x\|_M+\varepsilon<1$. For this $\varepsilon\in(0,\infty)$ there exists $\rho_0\in(0,\infty)$ such that $\rho_0 < \|x\|_M +\varepsilon<1$ and $\sum_{i=1}^dM(x_i/\rho_0) \leq 1$. But then, because $\rho_0<1$ and since $M$ in increasing, we obtain the contradiction
\[
1 = \sum_{i=1}^dM(x_i) < \sum_{i=1}^d M(x_i/\rho_0) \leq 1.
\]
Hence, we must have $\|x_M\|\geq 1$. Last but not least, we want to exclude that $\|x\|_M>1$, thereby establishing $\|x\|_M=1$. So assume that $\|x\|_M>1$. Then
\[
1 = \sum_{i=1}^d M(x_i) > \sum_{i=1}^d M(x_i/\|x\|_M).
\]
But then, because of continuity of $M$, there exists $\varepsilon=\varepsilon(x)\in(0,\infty)$ such that $\|x\|_M-\varepsilon>1$ and
\[
1 > \sum_{i=1}^d M(x_i/(\|x\|_M-\varepsilon)).
\]
This contradicts that $\|x\|_M$ is the infimum and so we must have $\|x\|_M=1$, which means $x\in\B_M^d$.
\end{proof}

As explained in the introduction, for our purposes it is crucial to understand the asymptotic volume of balls in Orlicz spaces. From a non-asymptotic point of view and for a fixed radius independent of the dimension, those volumes are known and follow from a more general result of Sch\"utt \cite{Sch1982} who obtained, among other things, simple formulas (up to absolute constants) for the volume of unit balls in finite-dimensional Banach spaces with a $1$-symmetric basis. More precisely, he proved that if $X$ is a finite-dimensional Banach space with a $1$-symmetric basis $e_1,\dots,e_n$ and norm $\|\cdot\|_X$, then
\[
\vol_d\Bigg(\Bigg\{ a\in\R^d \,:\, \Big\|\sum_{i=1}^da_ie_i\Big\|_X \leq 1\Bigg\}\Bigg) \approx 2^d \Big\|\sum_{i=1}^d e_i\Big\|_X^{-d}\,.
\]
The standard unit vectors in $\R^d$ form such a basis for the Orlicz spaces $\ell_M^d$ and therefore,
\[
\vol_d\big(\B_M^d\big) \approx 2^d \|(1)_{i=1}^d\|_M^{-d} = 2^{d}M^{-1}(1/d)^d.
\]
However, first of all such a bound is not sufficient for our purposes as it only provides estimates up to absolute constants and second, the natural setting to study the intersection of Orlicz balls (cf. \cite{SS1991,Schmu2001}) is to look at the volume of balls of dimension-dependent radius $dR$, i.e.,
\[
B_M^d(dR) := \Bigg\{x=(x_i)_{i=1}^d \in\R^d\,:\, \sum_{i=1}^d M(x_i) \leq dR \Bigg\},
\]
where $R\in(0,\infty)$. Note that for $R=1$ this is essentially the dimensional normalization considered in \cite{SS1991} as we may swallow the constant terms in the Orlicz function $M$. In the next section we present our result on the asymptotic (logarithmic) volume of Orlicz balls.

\subsection{Sharp Cram\'er large deviations \& Integral asymptotics}\label{sec:petrov}

In this subsection we shall briefly present the result of Petrov \cite{P1965} who proved a sharp version of Cram\'er's theorem (see, e.g., \cite{C1938, DZ2010}). We also wish to refer the reader to the recent work \cite{LR2020}, where sharp large deviations have been obtained in the geometric setting of $\ell_p^d$-spheres.

Consider a sequence $X_1,\dots,X_d$ of independent and identically distributed random variables with distribution $\Pro^X$ not concentrated on lattice. Define the set
\[
\mathscr B^+ := \Bigg\{ h \geq 0 \,:\, \int_{\R_{\geq 0}} e^{h x}\,\Pro^X(\dint x) < +\infty \Bigg\}.
\]
This set is non-empty, because we always have $0\in\mathscr B^+$. We let $B=\sup \mathscr B^+ \in[0,+\infty]$ and define for all $0<h<B$ the quantities
\[
R(h) := \int_\R e^{h x}\,\Pro^X(\dint x)
\]
and
\[
m(h) := \frac{1}{R(h)}\int_{\R} x e^{h x} \,\Pro^X(\dint x).
\]
Assume in the following that $B>0$.  Then one can show (see \cite[Lemma, pp. 287]{P1965}) the following limit exists,
\[
A_0:=\lim_{h\uparrow B}m(h).
\]
Moreover, if $\E[X_1]>-\infty$ and $A_0<+\infty$, then \cite[Theorem 1]{P1965} states that, as $d\to\infty$,
\[
\Pro\Bigg[\frac{1}{d}\sum_{i=1}^d X_i \geq x\Bigg] = \frac{1}{h_* \sqrt{2\pi d \sigma^2(h_*)}}e^{d\log R(h_*) - d h_* x}(1+o(1))
\]
where the convergence is uniform on
\[
\E[X_1]+\varepsilon \leq x \leq A_0-\varepsilon,\qquad \varepsilon\in(0,\infty)
\]
and $h_*$ is the unique real solution to $m(h_*) = x$ while $\sigma^2(h_*) = \frac{d}{dh} m(h)|_{h=h_*}$. In the proof, after a suitable measure tilting guaranteeing the existence of all exponential moments, Petrov uses the Berry-Esseen theorem to determine the asymptotic of the integral (see \cite[Equation (4.11)]{P1965})
\begin{align}\label{eq:Petrov integral}
I:= \int_0^\infty e^{-h_* t \sigma(h_*)\sqrt{d}}\,\dint \overline{F}_d(t),
\end{align}
where $\overline{F}_d$ is the distribution function of the random variables
\[
\frac{\sum_{i=1}^d \overline{X}_i-dm(h_*)}{\sigma(h_*)\sqrt{d}},
\]
with $\overline{X}_1$ having distribution function
\[
x\mapsto \frac{1}{R(h_*)}\int_{-\infty}^x e^{h_*y}\,\Pro^X(\dint y).
\]
What Petrov obtains (see \cite[Equation (4.19)]{P1965}) is that
\begin{align}\label{eq:asymptotic of Petroc integral}
I = \frac{1}{h_* \sigma(h_*) \sqrt{2\pi d}}(1+o(1)).
\end{align}
Exactly such an integral is what appears in the computation of the precise asymptotic volume of Orlicz balls and we can therefore use the asymptotic in \eqref{eq:asymptotic of Petroc integral}. 

\section{The asymptotic volume of Orlicz balls}\label{sec:asymptoticvolume}

We shall now present the computation of the volume of Orlicz balls $B_M^d(dR)$, $R\in(0,\infty)$. In fact, we start with the asymptotic \emph{logarithmic} volume. The proof of this result is based on an exponential tilting technique known, for instance, from large deviations theory, which is then coupled with the classical central limit theorem. Independently, this result was obtained in the updated version of \cite{KLR2019} in the context of the asymptotic thin-shell condition for Orlicz balls using a large deviations approach, while the authors finalized this manuscript.

\begin{proposition}\label{prop:logvolume}
Let $d\in\N$, $R\in(0,\infty)$, and $M$ be an Orlicz function. Then, as $d\to\infty$,
\[
\vol_d\big(B_M^d(dR)\big)^{1/d} \to e^{\varphi(\alpha_*)-\alpha_* R},
\]
that is, on a logarithmic scale, we have
\[
\lim_{d\to\infty}\frac 1d \log \,\vol_d\big(B_M^d(dR)\big) = \varphi(\alpha_*)-\alpha_* R,
\]
where $\varphi:(-\infty,0)\to \R$ is given by $\varphi(\alpha) = \log \int_{\R}e^{\alpha M(x)}\,\dint x$ and $\alpha_*<0$ is chosen in such a way that $\varphi'(\alpha_*) = R$.
\end{proposition}
\begin{proof}
Let us define for $\alpha<0$ the function
\[
\varphi(\alpha) := \log \int_{\R} e^{\alpha M(x)}\,\dint x,
\]
which is finite because $M$ is an Orlicz function (see \eqref{eq:integral finite}).
Moreover, $\varphi$ is twice continuously differentiable on $(-\infty,0)$ with
\begin{equation}\label{eq:derivative of phi}
\varphi'(\alpha) = \frac{\int_{\R}\partial_\alpha e^{\alpha M(x)}\,\dint x }{\int_{\R}e^{\alpha M(x)}\,\dint x} = \frac{\int_{\R} M(x) e^{\alpha M(x)}\,\dint x }{\int_{\R}e^{\alpha M(x)}\,\dint x}\,.
\end{equation}
Let $\alpha_*:=\alpha_*(R)<0$ be such that $\varphi'(\alpha_*) = R$.  To see that such $\alpha_*$ exists uniquely,  one can easily check that $\lim_{\alpha \uparrow 0} \varphi'(\alpha) =+\infty$ and $\lim_{\alpha \to -\infty} \varphi'(\alpha) = 0$. Since the function $\varphi'$ is strictly monotone increasing, $\alpha_*$ with the required property exists and is unique.    Now we consider independent and identically distributed random variables $Z_1,Z_2,\dots$ with Lebesgue-density given by
\[
p(x) = e^{\alpha_*M(x) - \varphi(\alpha_*)}, \qquad x\in\R.
\]
This is indeed a density because of the definition of $\varphi$, since
\[
\int_{\R}e^{\alpha_*M(x) - \varphi(\alpha_*)} \,\dint x = e^{- \varphi(\alpha_*)} \int_{\R }e^{\alpha_*M(x)}\,\dint x = e^{- \varphi(\alpha_*)} e^{ \varphi(\alpha_*)} = 1.
\]
We now show that $\E[M(Z_1)] = R$ and $\Var[M(Z_1)] = \varphi''(\alpha_*)>0$. For this consider the function
\[
(-\infty,0) \ni \alpha\mapsto e^{\varphi(\alpha)} = \int_{\R}e^{\alpha M(x)}\,\dint x.
\]
Then, using \eqref{eq:derivative of phi}, we find that
\begin{equation}\label{eq: first derivative of exp phi}
\frac{d}{d\alpha} e^{\varphi(\alpha)} = \varphi'(\alpha) e^{\varphi(\alpha)} = \int_{\R}M(x)e^{\alpha M(x)}\,\dint x.
\end{equation}
Similarly, we obtain
\begin{equation}\label{eq: second derivative of exp phi}
\frac{d^2}{d\alpha^2} e^{\varphi(\alpha)} = e^{\varphi(\alpha)}\varphi''(\alpha) + e^{\varphi(\alpha)}(\varphi'(\alpha))^2 = \int_{\R} M(x)^2e^{\alpha M(x)}\,\dint x.
\end{equation}
Therefore,
\[
\E[M(Z_1)] = \int_{\R}M(x)e^{\alpha_*M(x)-\varphi(\alpha_*)}\,\dint x = e^{-\varphi(\alpha_*)}\int_{\R} M(x)e^{\alpha_*M(x)}\,\dint x \stackrel{\eqref{eq: first derivative of exp phi}}{=} e^{-\varphi(\alpha_*)} \varphi'(\alpha_*)e^{\varphi(\alpha_*)} = \varphi'(\alpha_*) = R
\]
and
\[
\E[M(Z_1)^2] = \int_{\R}e^{\alpha_*M(x) - \varphi(\alpha_*)}M(x)^2\,\dint x = e^{-\varphi(\alpha_*)} \int_{\R}e^{\alpha_*M(x)}M(x)^2 \,\dint x \stackrel{\eqref{eq: second derivative of exp phi}}{=} \varphi''(\alpha_*) + \varphi'(\alpha_*)^2.
\]
Hence, we find
\[
\Var[M(Z_1)] = \varphi''(\alpha_*).
\]

Now we consider the independent and identically distributed random variables $Y_i := M(Z_i)-R$, $i\in\N$ with $\E[Y_1] = 0$ and $\Var[Y_1] =\Var[M(Z_1)] = \varphi''(\alpha_*) =: \sigma_*^2 $. Then
\begin{align*}
\vol_d\big(B_M^d(dR)\big) & = \int_{\R^d} \mathbbm 1_{B_M^d(dR)} (x_1,\dots,x_d)\, d\lambda^d(x_1\dots,x_d) \cr
& = \int_{\R^d}  \mathbbm 1_{B_M^d(dR)} (x_1,\dots,x_d) \underbrace{e^{-\alpha_*\sum_{i=1}^dM(x_i) + d\varphi(\alpha_*)} \prod_{i=1}^dp(x_i)}_{=1}\, d\lambda(x_1)\dots d\lambda(x_d) \cr
& = \E\Big[\mathbbm 1_{B_M^d(dR)}(Z_1,\dots,Z_d) e^{-\alpha_*\sum_{i=1}^dM(Z_i)+d\varphi(\alpha_*)}\Big] \cr
& = \E\Big[\mathbbm 1_{\{\sum_{i=1}^dY_i \leq 0\}} e^{-\alpha_*\sum_{i=1}^dY_i - d\alpha_*R+d\varphi(\alpha_*)}\Big] \cr
& = e^{d\big(\varphi(\alpha_*)-\alpha_* R\big)} \E\Big[ \mathbbm 1_{\{\sum_{i=1}^dY_i \leq 0\}} e^{-\alpha_*\sum_{i=1}^dY_i}\Big].
\end{align*}

Let us continue with a lower and an upper bound. We have, for every $c\in(0,\infty)$,
\[
\E\Big[ \mathbbm 1_{\big\{\sum_{i=1}^dY_i \leq 0\big\}} e^{-\alpha_*\sum_{i=1}^dY_i}\Big] \geq \E \Big[ \mathbbm 1_{\big\{-c\sqrt{d} \leq \sum_{i=1}^dY_i \leq 0\big\}} e^{-\alpha_*\sum_{i=1}^dY_i} \Big] \geq \E \Big[ \mathbbm 1_{\big\{-c\sqrt{d} \leq \sum_{i=1}^dY_i \leq 0\big\}} e^{c \alpha_*\sqrt{d}} \Big],
\]
where we used that $-\alpha_*>0$. For the last expression, we have
\[
\E \Big[ \mathbbm 1_{\big\{-c\sqrt{d} \leq \sum_{i=1}^dY_i \leq 0\big\}} e^{c \alpha_*\sqrt{d}} \Big] = e^{c \alpha_*\sqrt{d}} \Pro\Big[ \frac{1}{\sqrt{d}}\sum_{i=1}^dY_i\in[-c,0]\Big]
\]
and so
\[
\vol_d\big(B_M^d(dR)\big) \geq e^{d\big(\varphi(\alpha_*)-\alpha_* R\big)} e^{c \alpha_*\sqrt{d}} \Pro\Big[ \frac{1}{\sqrt{d}}\sum_{i=1}^dY_i\in[-c,0]\Big].
\]
Note that by the central limit theorem,
\[
\Pro\Big[ \frac{1}{\sqrt{d}}\sum_{i=1}^dY_i\in[-c,0]\Big] \stackrel{d\to\infty}{\longrightarrow} \mathscr N(0,\sigma_*^2)([0,c]),
\]
which is a strictly positive constant.
Similar to the lower bound, we can obtain an upper one. We have
\[
\E\Big[ \mathbbm 1_{\big\{\sum_{i=1}^dY_i \leq 0\big\}} e^{-\alpha_*\sum_{i=1}^dY_i}\Big] \leq \E\Big[ \mathbbm 1_{\big\{\frac{1}{\sqrt{d}}\sum_{i=1}^dY_i \leq 0\big\}} e^{-\alpha_*\sum_{i=1}^dY_i}\Big] \leq \Pro\Big[ \frac{1}{\sqrt{d}}\sum_{i=1}^dY_i\in (-\infty,0]\Big]
\]
and therefore,
\[
\vol_d\big(B_M^d(dR)\big) \leq e^{d\big(\varphi(\alpha_*)-\alpha_* R\big)} \Pro\Big[ \frac{1}{\sqrt{d}}\sum_{i=1}^dY_i\in (-\infty,0]\Big].
\]
Again, the central limit theorem implies that
\[
\Pro\Big[ \frac{1}{\sqrt{d}}\sum_{i=1}^dY_i\in (-\infty,0]\Big]\stackrel{d\to\infty}{\longrightarrow} \mathscr N(0,\sigma_*^2)((-\infty,0]) = \frac{1}{2}.
\]
Collecting what we obtained above, we see that, for any $c\in(0,\infty)$,
\[
e^{d\big(\varphi(\alpha_*)-\alpha_* R\big)} e^{c \alpha_*\sqrt{d}} \Pro\Big[ \frac{1}{\sqrt{d}}\sum_{i=1}^dY_i\in[-c,0]\Big] \leq \vol_d\big(B_M^d(dR)\big) \leq e^{d\big(\varphi(\alpha_*)-\alpha_* R\big)} \Pro\Big[ \frac{1}{\sqrt{d}}\sum_{i=1}^dY_i\in (-\infty,0]\Big].
\]
Now, taking the $d^{\text{th}}$ root and letting $d\to\infty$, the central limit theorem shows that
\[
\vol_d\big(B_M^d(dR)\big)^{1/d} \to e^{\varphi(\alpha_*)-\alpha_* R},
\]
which completes the proof.
\end{proof}

The next result provides the exact asymptotics for the volume. Its proof is more delicate and based on ideas that can be found in a paper on large deviations for sums of independent and identically distributed random variables by Petrov \cite{P1965}, more precisely, in the proof of Theorem 1 there. Theorem 1 is a version of Cram\'er's theorem \cite{C1938} (see also \cite{DZ2010}), but not just on a logarithmic scale, providing the precise asymptotics. We shall use parts of his result as outlined in Section \ref{sec:petrov}. Let us also remark that some of the results in \cite{P1965} can be found in the earlier work \cite{BRR1960} of Bahadur and Rao of which Petrov was unaware.

\begin{proposition}\label{prop:asymptotic_volume}
Let $d\in\N$, $R\in(0,\infty)$, and $M$ be an Orlicz function. Then, as $d\to\infty$,
\[
\vol_d\big(B_M^d(dR)\big) \sim \frac{1}{|\alpha_*| \sqrt{2\pi d\, \sigma_*^2}  }e^{d[\varphi(\alpha_*)-\alpha_* R]},
\]
where $\varphi:(-\infty,0)\to \R$ is given by $\varphi(\alpha) = \log \int_{\R}e^{\alpha M(x)}\,\dint x$ and $\alpha_*<0$ is chosen in such a way that $\varphi'(\alpha_*) = R$.
\end{proposition}

\begin{proof}
We use the notation as introduced in the proof of Proposition \ref{prop:logvolume}. As shown there, we have
\begin{align}\label{eq:volume of B_M^d(dR)}
\vol_d\big(B_M^d(dR)\big)
& = e^{d\big(\varphi(\alpha_*)-\alpha_* R\big)} \E\Big[ \mathbbm 1_{\{\sum_{i=1}^dY_i \leq 0\}} e^{-\alpha_*\sum_{i=1}^dY_i}\Big],
\end{align}
where $Y_i:= M(Z_i)-R$, $i\in\{1,\dots,d\}$ with $Z_1,\dots, Z_d$ independent and having Lebesgue density $p(x) = \exp(\alpha_*M(x)-\varphi(\alpha_*))$, $x\in\R$.  In particular, the distribution of $Y_1$ is not concentrated on a lattice.
Recall also that $\E[Y_1]=0$ and  $\Var[Y_1] = \varphi''(\alpha_*) =: \sigma_*^2$.
Let us denote by $\mu_d$ the distribution of $\sum_{i=1}^dY_i$. Then we have
\[
\E\Bigg[ \mathbbm 1_{\{\sum_{i=1}^dY_i \leq 0\}} e^{-\alpha_*\sum_{i=1}^dY_i}\Bigg] = \int_{-\infty}^0 e^{-\alpha_* y}\,\mu_d(\dint y).
\]
Now it is left to understand the integral on the right-hand side. Here we observe that it is exactly the integral $I$ that appears in \cite[Equation 4.11]{P1965} (see also \eqref{eq:Petrov integral}), just with a different sign. As is demonstrated in Petrov's proof via a Berry--Esseen argument, for $d\to\infty$ this integral can be evaluated (see \cite[Equation 4.19]{P1965} and \eqref{eq:asymptotic of Petroc integral}) as follows,
\[
\int_{-\infty}^0 e^{-\alpha_* y}\,\mu_d(\dint y) \sim \frac{1}{|\alpha_*|\sqrt{2\pi d \sigma_*^2}}.
\]
In combination with \eqref{eq:volume of B_M^d(dR)}, we obtain for $d\to\infty$ that
\[
\vol_d\big(B_M^d(dR)\big) \sim \frac{1}{|\alpha_*| \sqrt{2\pi d\, \sigma_*^2}  }e^{d[\varphi(\alpha_*)-\alpha_* R]},
\]
which completes the proof.
\end{proof}

\begin{rmk}
Let us note again that the result of Proposition \ref{prop:asymptotic_volume} holds in a more general setting than the one for Orlicz functions presented here. In fact, we only need that the random variables $Y_1,\dots,Y_d$ which are defined in terms of $M$ do not have a distribution concentrated on some lattice together with the assumption that for $M$ the integral in \eqref{eq:integral finite} is finite (for some $a>0$) and that we can find $\alpha_*<0$ such that $\varphi'(\alpha_*) = R$.
\end{rmk}

\section{The volume of intersections of Orlicz balls}\label{sec:volumeintersection}

Equipped with the asymptotics of the ($\log$-)volume of Orlicz balls, we can now study the volume of the intersection of two such balls as the dimension of the ambient space tends to infinity. In fact, we shall see what the precise threshold for the change in convergence behavior is and obtain the corresponding dichotomy previously known for $\ell_p^d$-balls through the work of Schechtman and Schmuckenschl\"ager \cite{SS1991}. The precise quantity of interest is
\[
\frac{\vol_d\big(B_{M_{1}}^d(dR_1) \cap B_{M_2}^d(dR_2)\big)}{\vol_d\big(B_{M_1}^d(dR_1)\big)}, \qquad d\in\N,
\]
where $M_1$ and $M_2$ are Orlicz functions and $R_1, R_2\in(0,\infty)$.
This is indeed a canonical way of generalizing the framework studied in \cite{SS1991} as it resembles the uniform distribution on $B_{M_1}^d(dR_1)$.


\subsection{Volume of intersections -- Non-critical case}

The general idea of proof is to proceed in a similar way to when we determined the $\log$-asymptotic volume of Orlicz balls $B_M^d(dR)$. However, one needs to consider two appropriate sets of random variables and complement the argument by Cram\'er's large deviation theorem for independent and identically distributed random variables.

\begin{proof}[Proof of Theorem \ref{thm:dichotomy}]
Let $Z_1,\dots,Z_d$ be independent random variables each with distribution given by the Lebesgue-density
\[
p_1(x) := e^{\alpha_*M_1(x) - \varphi_1(\alpha_*)}, \qquad x\in\R,
\]
where for $\alpha<0$, we have
\[
\varphi_1(\alpha) := \log \int_{\R}e^{\alpha M_1(x)}\,\dint x
\]
and where $\alpha_*<0$ is now chosen such that $\varphi_1'(\alpha_*) = R_1$.
Then, we find that
\begin{align*}
& \vol_d\big(B_{M_1}^d(dR_1)\cap B_{M_2}^d(dR_2)\big) \cr
& = \int_{\R^d} \mathbbm 1_{B_{M_1}^d(dR_1)} (x_1,\dots,x_d)\mathbbm 1_{B_{M_2}^d(dR_2)} (x_1,\dots,x_d)\, d\lambda^d(x_1\dots,x_d) \cr
& = \int_{\R^d}  \mathbbm 1_{B_{M_1}^d(dR_1)} (x_1,\dots,x_d) \mathbbm 1_{B_{M_2}^d(dR_2)}(x_1,\dots,x_d)  e^{-\alpha_*\sum_{i=1}^dM_1(x_i) + d\varphi_1(\alpha_*)} \prod_{i=1}^dp_1(x_i)\, d\lambda(x_1)\dots d\lambda(x_d) \cr
& = \E\Big[\mathbbm 1_{B_{M_1}^d(dR_1)}(Z_1,\dots,Z_d) \mathbbm 1_{B_{M_2}^d(dR_2)}(Z_1,\dots,Z_d) e^{-\alpha_*\sum_{i=1}^dM_1(Z_i)+d\varphi_1(\alpha_*)}\Big].
\end{align*}
Now we need to modify the probabilistic argument seen before. Let us define random variables $Y_1^{(1)},\dots,Y_d^{(1)}$ and $Y_1^{(2)},\dots,Y_d^{(2)}$ via
\[
Y_i^{(1)} := M_1(Z_i) - R_1 \qquad\text{and}\qquad Y_i^{(2)}:= M_2(Z_i) - \int_{\R}M_2(x)p_1(x)\,\dint x
\]
for $i\in\{1,\dots,d\}$. Then $Y_1^{(1)},\dots,Y_d^{(1)}$ are independent and also $Y_1^{(2)},\dots,Y_d^{(2)}$ are independent. Moreover, $\E[Y_1^{(1)}] = 0 = \E[Y_1^{(2)}]$ and $\Var[Y_1^{(1)}] = \varphi_1''(\alpha_*)$ while $\Var[Y_1^{(2)}] = \E[(Y_1^{(2)})^2] = \Var[M_2(Z_1)]$. Using those transformations of the original random variables, we may write
\begin{align*}
\vol_d\big(B_{M_1}^d(dR_1)\cap B_{M_2}^d(dR_2)\big) & = \E\Bigg[\mathbbm 1_{\big\{\sum_{i=1}^dY_i^{(1)}\leq 0 \big\}} \mathbbm 1_{\big\{\sum_{i=1}^dY_i^{(2)} \leq d[R_2 - \int_{\R}M_2(x)p_1(x)\,\dint x] \big\}} e^{-\alpha_*\sum_{i=1}^dY_i^{(1)} - d\alpha_*R_1+d\varphi_1(\alpha_*)}\Bigg] \cr
& =e^{d[\varphi_1(\alpha_*) - \alpha_*R_1]} \E\Bigg[\mathbbm 1_{\big\{\sum_{i=1}^dY_i^{(1)}\leq 0 \big\}} \mathbbm 1_{\big\{\sum_{i=1}^dY_i^{(2)} \leq d[R_2 - \int_{\R}M_2(x)p_1(x)\,\dint x] \big\}} e^{-\alpha_*\sum_{i=1}^dY_i^{(1)}}\Bigg].
\end{align*}
The idea is that  by the (strong) law of large numbers,
\[
\frac{1}{d}\sum_{i=1}^d Y_i^{(2)}\stackrel{a.s.}{\longrightarrow} 0\qquad\text{as }\,d\to\infty.
\]
Hence, we see that for $d\to\infty$ the event
\[
\Bigg\{\sum_{i=1}^dY_i^{(2)} \leq d\Big[R_2 - \int_{\R}M_2(x)p_1(x)\,\dint x\Big] \Bigg\}
\]
occurs with probability converging to $0$ -- even exponentially fast -- if $\int_{\R}M_2(x)p_1(x)\,\dint x>R_2$ and probability approaching $1$ if
$\int_{\R}M_2(x)p_1(x)\,\dint x<R_2$.

Consider the case $\int_{\R}M_2(x)p_1(x)\,\dint x>R_2$. To obtain an upper bound for the volume of the intersection, observe that
$$
\E\Bigg[\mathbbm 1_{\big\{\sum_{i=1}^dY_i^{(1)}\leq 0 \big\}} \mathbbm 1_{\big\{\sum_{i=1}^dY_i^{(2)} \leq d[R_2 - \int_{\R}M_2(x)p_1(x)\,\dint x] \big\}} e^{-\alpha_*\sum_{i=1}^dY_i^{(1)}}\Bigg]
\leq
\Pro\Bigg[\sum_{i=1}^dY_i^{(2)} \leq d[R_2 - \int_{\R}M_2(x)p_1(x)\,\dint x]\Bigg],
$$
which goes to $0$ exponentially fast. It follows from the lower bound for $\vol_d\big(B_{M_1}^d(dR_1)\big)$ in the proof of Theorem~\ref{thm:log-volume orlicz} that
$$
\frac{\vol_d\big(B_{M_1}^d(dR_1)\cap B_{M_2}^d(dR_2)\big)}  {\vol_d\big(B_{M_1}^d(dR_1)\big)}
$$
goes to $0$ exponentially fast by Cram\'er's theorem (see \cite{DZ2010}) as $d\to\infty$.

Now consider the case $\int_{\R}M_2(x)p_1(x)\,\dint x < R_2$. We have
$$
\frac{\vol_d\big(B_{M_1}^d(dR_1)\cap B_{M_2}^d(dR_2)\big)}  {\vol_d\big(B_{M_1}^d(dR_1)\big)}
=
1- \frac{ \E\Bigg[\mathbbm 1_{\big\{\sum_{i=1}^dY_i^{(1)}\leq 0 \big\}} \mathbbm 1_{\big\{\sum_{i=1}^dY_i^{(2)} > d[R_2 - \int_{\R}M_2(x)p_1(x)\,\dint x] \big\}} e^{-\alpha_*\sum_{i=1}^dY_i^{(1)}}\Bigg]}{ \E\Bigg[\mathbbm 1_{\big\{\sum_{i=1}^dY_i^{(1)}\leq 0 \big\}} e^{-\alpha_*\sum_{i=1}^dY_i^{(1)}}\Bigg]}.
$$
We have to show that the quotient of expectations on the right-hand side goes to $0$. To this end, observe that the expectation in the numerator can be estimated as follows:
$$
\E\Bigg[\mathbbm 1_{\big\{\sum_{i=1}^dY_i^{(1)}\leq 0 \big\}} \mathbbm 1_{\big\{\sum_{i=1}^dY_i^{(2)} > d[R_2 - \int_{\R}M_2(x)p_1(x)\,\dint x] \big\}} e^{-\alpha_*\sum_{i=1}^dY_i^{(1)}}\Bigg]
\leq
\Pro\Bigg[\sum_{i=1}^dY_i^{(2)} > d[R_2 - \int_{\R}M_2(x)p_1(x)\,\dint x]\Bigg],
$$
which goes to $0$ exponentially fast by Cram\'er's theorem (see \cite{DZ2010}) as $d\to\infty$. Applying to the expectation in the denominator the lower bound from the proof of Theorem~\ref{thm:log-volume orlicz}, we obtain that
$$
\frac{\vol_d\big(B_{M_1}^d(dR_1)\cap B_{M_2}^d(dR_2)\big)}  {\vol_d\big(B_{M_1}^d(dR_1)\big)}
$$
goes to $1$ exponentially fast.

Overall, the above yields the dichotomy
\[
\frac{\vol_d\big(B_{M_{1}}^d(dR_1) \cap B_{M_2}^d(dR_2)\big)}{\vol_d\big(B_{M_1}^d(dR_1)\big)} \stackrel{d\to\infty}{\longrightarrow}
\begin{cases}
0 & :\, \int_{\R}M_2(x)p_1(x)\,\dint x>R_2 \\
1 & :\, \int_{\R}M_2(x)p_1(x)\,\dint x<R_2,
\end{cases}
\]
which completes the proof.
\end{proof}


\section{The asymptotic volume ratio of $2$-concave Orlicz spaces}\label{sec:asymptotic volume ratio}

We are now going to determine the asymptotic volume ratio of $2$-concave Orlicz spaces, i.e., those defined by Orlicz functions where $M\circ \sqrt{|\cdot|}$ is a concave function on $\R$. This is, for instance, the case when $M(t)=|t|^p$ for $1\leq p \leq 2$. Since we already obtained the precise asymptotic volume of Orlicz balls, we merely need to determine the John ellipsoid, i.e., the maximum volume ellipsoid in $B_M^d(d)$. To do this, we recall from \cite[Section 16]{TJ1988} that a Banach space $X$ is said to have \textit{enough symmetries} if the only linear operators that commute with every isometry of $X$ are multiples of the identity. If $X$ is $d$-dimensional and has enough symmetries, it is known that $\mathscr{E}_X$ is a suitable multiple of the Euclidean unit ball of the same dimension. More precisely,
\begin{equation}\label{eq:JohnEllipsoid}
\mathscr{E}_X = \big\|{\rm id}:\ell_2^d\to X\big\|^{-1}\mathbb{B}_2^d,
\end{equation}
where $\ell_2^d$ is the $d$-dimensional Euclidean space with the Euclidean unit ball $\mathbb{B}_2^d$ and ${\rm id}:\ell_2^d\to X$ stands for the identity operator from $\ell_2^d$ to $X$ with the standard operator norm $\|{\rm id}:\ell_2^d\to X\|$ (see, e.g., \cite{DM2006}). Orlicz sequence spaces are Banach spaces with a $1$-symmetric basis (where the norm is invariant under permutations and signs) and have enough symmetries \cite{DM2006,TJ1988}.

Observe that by Lemma \ref{lem:equality of unit balls} the Banach space $\ell_{M/d}^d$ has unit ball
\[
B_M^d(d) = \Bigg\{ x=(x_1,\dots,x_d)\in\R^d\,:\, \frac{1}{d}\sum_{i=1}^d M(x_i) \leq 1 \Bigg\}.
\]
The following result contains the precise asymptotic volume ratio of the space $\ell_{M/d}^d$, when the defining Orlicz function $M$ is $2$-concave. 

\begin{proposition}
Let $M$ be a $2$-concave Orlicz function. Then, as $d\to\infty$, we have
\[
{\rm vr}\big(B_M^d(d)\big) \sim \frac{1}{\sqrt{2\pi e}M^{-1}(1)}e^{\varphi(\alpha_*) - \alpha_*},
\]
where $\varphi:(-\infty,0)\to \R$ is given by $\varphi(\alpha) = \log \int_{\R}e^{\alpha M(x)}\,\dint x$ and $\alpha_*<0$ is chosen in such a way that $\varphi'(\alpha_*) = 1$.
\end{proposition}
\begin{proof}
In view of \eqref{eq:JohnEllipsoid}, Proposition \ref{prop:asymptotic_volume}, and the fact that
\[
\vol_d\big(\B_2^d\big)
= \frac{\sqrt{\pi}^d}{\Gamma(1+d/2)} \stackrel{d\to\infty}{\sim} \frac{1}{\sqrt{d\pi}} \Bigg(\frac{2\pi e}{d}\Bigg)^{d/2},
\]
we merely need to compute the operator norm of the natural embedding of $\ell_2^d$ into $\ell_{M/d}^d$. Note that
\begin{align*}
\|\id:\ell_2^d\to \ell_{M/d}^d\|
=  \sup_{(x_1,\dots,x_d)\in \SSS^{d-1}_2} \|(x_1,\dots,x_d)\|_{M/d}
 = \sup_{(x_1,\dots,x_d)\in \SSS^{d-1}_2}  \inf\Bigg\{\rho\in(0,\infty)\,:\, \sum_{i=1}^d \frac{1}{d}M\Big(\frac{x_i}{\rho}\Big) \leq 1 \Bigg\}.
\end{align*}
Now let us assume that $(x_1,\dots,x_d)\in\SSS_2^{d-1}$. Then it follows from the concavity of $M\circ\sqrt{|\cdot|}:\R\to[0,\infty)$ that
\[
\sum_{i=1}^d \frac{1}{d}M\Big(\frac{x_i}{\rho}\Big) = \sum_{i=1}^d \frac{1}{d}(M\circ\sqrt{|\cdot|})\Big(\frac{x_i^2}{\rho^2}\Big) \leq (M\circ\sqrt{|\cdot|})\Bigg(\sum_{i=1}^d \frac{x_i^2}{d\rho^2}\Bigg) = (M\circ\sqrt{|\cdot|})\Big(\frac{1}{d\rho^2}\Big) = M\Big(\frac{1}{\sqrt{d}\rho}\Big).
\]
This means that
\[
\|\id:\ell_2^d\to \ell_{M/d}^d\| \leq \sup_{(x_1,\dots,x_d)\in \SSS^{d-1}_2}  \inf\Bigg\{\rho\in(0,\infty)\,:\,  M\Big(\frac{1}{\sqrt{d}\rho}\Big) \leq 1 \Bigg\} = \frac{1}{\sqrt{d}M^{-1}(1)}.
\]
The lower bound follows from considering the vector $x_0 = (1/\sqrt{d},\dots,1/\sqrt{d})\in\R^d$, namely
\[
\|\id:\ell_2^d\to \ell_{M/d}^d\| \geq \big\|(1/\sqrt{d},\dots,1/\sqrt{d})\big\|_{M/d} = \frac{1}{\sqrt{d}}\|(1,\dots,1)\|_{M/d} = \frac{1}{\sqrt{d}M^{-1}(1)}.
\]
Therefore, we have $\|\id:\ell_2^d\to \ell_{M/d}^d\|^{-1} = \sqrt{d}M^{-1}(1)$ and the John ellipsoid in $\ell_{M/d}^d$ is hence given by
\[
\mathscr E_M := \sqrt{d}M^{-1}(1)\,\B_2^d.
\]
This means that, as $d\to\infty$,
\begin{align*}
{\rm vr}\big(B_M^d(d)\big) & = \frac{1}{\sqrt{d}M^{-1}(1)}\, \frac{\vol_d\big(B_M^d(d)\big)^{1/d}}{\vol_d\big(\B_2^d\big)^{1/d}} \sim \frac{1}{\sqrt{d}M^{-1}(1)}\, \frac{\Bigg(\frac{1}{|\alpha_*| \sqrt{2\pi d\, \sigma_*^2}  }e^{d[\varphi(\alpha_*)-\alpha_*] }\Bigg)^{1/d}}{\Bigg(\frac{1}{\sqrt{d\pi}} \Bigg(\frac{2\pi e}{d}\Bigg)^{d/2}\Bigg)^{1/d}} \cr
& \sim \frac{\sqrt{d}}{\sqrt{2\pi e}\sqrt{d}M^{-1}(1)}e^{\varphi(\alpha_*) - \alpha_*} = \frac{1}{\sqrt{2\pi e}M^{-1}(1)}e^{\varphi(\alpha_*) - \alpha_*},
\end{align*}
which completes the proof.
\end{proof}

\subsection*{Acknowledgement}
ZK has been supported by the German Research Foundation under Germany’s Excellence Strategy EXC 2044 – 390685587, Mathematics M\"unster: Dynamics - Geometry - Structure. JP is supported by the Austrian Science Fund (FWF) Project P32405 \textit{Asymptotic geometric analysis and applications}.
We thank Chiranjib Mukherjee for pointing out Reference~\cite{BS2016} to us.

\bibliographystyle{plain}
\bibliography{vio}

\bigskip
\bigskip
	
	\medskip
	
	\small

	\noindent \textsc{Zakhar Kabluchko:} Faculty of Mathematics, University of M\"unster, Orl\'eans-Ring 10,
		48149 M\"unster, Germany
		
	\noindent
		{\it E-mail:} \texttt{zakhar.kabluchko@uni-muenster.de}
	
		\medskip
	
	\noindent \textsc{Joscha Prochno:} Institute of Mathematics and Scientific Computing,
	University of Graz, Heinrichstrasse 36, 8010 Graz, Austria
	
	\noindent
	{\it E-mail:} \texttt{joscha.prochno@uni-graz.at}

\end{document}